\documentclass{amsart}
\usepackage{psfrag,graphicx,epsf}
\usepackage{amsmath,amsthm,epsf,amscd,amssymb,verbatim}
\usepackage{palatino}
\usepackage[mathbf,mathcal]{euler}

\newtheorem{theorem}{Theorem}
\newtheorem{lemma}[theorem]{Lemma}
\newtheorem{cor}[theorem]{Corollary}
\newtheorem{prop}[theorem]{Proposition}

\theoremstyle{definition}
\newtheorem{definition}[theorem]{Definition}
\theoremstyle{remark}
\newtheorem{remark}[theorem]{Remark}
\newtheorem{ex}[theorem]{Example}

\numberwithin{equation}{section} \font\bbb=msbm10 scaled 1100

\newcommand{\ncmnd}{\newcommand}
\ncmnd{\nthm}{\newtheorem}

\nthm{thm}[theorem]{Theorem}
\nthm{lem}[theorem]{Lemma}

\theoremstyle{definition}
\nthm{defn}[theorem]{Definition}

\theoremstyle{remark}
\nthm{rem}[theorem]{Remark}
\nthm{examp}[theorem]{Example}

\newcommand{\R}{\mbox{\bbb R}}

\newcommand{\real}{\mbox{\bbb R}}

\newcommand{\euc}{\mbox{\bbb E}}
\newcommand{\rest}[2]{{\left.{#1}\right\vert_{{#2}}}}
\newcommand{\Domain}{{\mathcal D}}      

\newcommand{\catzero}{{\mbox{\sc cat($0$)}}}
\newcommand{\catk}{{\mbox{\sc cat($K$)}}}
\newcommand{\catone}{{\mbox{\sc cat($1$)}}}

\begin{document}

\title[TOTAL CURVATURE AND SIMPLE PURSUIT] {TOTAL CURVATURE AND SIMPLE PURSUIT ON DOMAINS
OF CURVATURE BOUNDED ABOVE}
\author{S. Alexander}
\address{Department of Mathematics,
        University of Illinois, Urbana IL, 61801}
\email{sba@math.uiuc.edu}

\author{R. Bishop}
\address{Department of Mathematics,
        University of Illinois, Urbana IL, 61801}
\email{bishop@math.uiuc.edu}

\author{R. Ghrist}
\address{Departments of Mathematics and Electrical/Systems Engineering,
        University of Pennsylvania, Philadelphia PA, 19104}
\email{ghrist@math.upenn.edu}

\begin{abstract}
We show how circumradius and asymptotic behavior of curves in
\catzero\ and \catk\,  spaces ($K>0$) are controlled by growth rates
of total curvature.  We apply our results to pursuit and evasion games
of capture type with simple pursuit motion, generalizing results that
are known for convex Euclidean domains, and obtaining  results that
are new for convex Euclidean domains and hold on playing fields
vastly more general than these.
\end{abstract}



\maketitle

\section{Introduction}
\label{sec_Intro}

The goals of this paper are twofold:
\begin{enumerate}
\item We study {\em total curvature} of a curve (the integral of its
    curvature) in spaces of curvature bounded above, and relate the
    total curvature, the curve's circumradius function, the
    asymptotic behavior of the curve, and the domain's curvature
    bound.
\item We apply these results to a foundational problem in {\em
    pursuit-evasion games}, where an evader moves in a domain and
    is followed by a pursuer along a {\em pursuit curve}. We study
    the {\em capture problem:} whether the pursuer ever catches
    (comes sufficiently close to) the evader. Although our total
    curvature and circumradius results are new even for convex
    Euclidean playing fields, our playing fields are vastly more
    general than these.
\end{enumerate}

We assume that the reader is familiar with the basic notions of \catk\
and Alexandrov geometry (see, e.g., \cite{BH99,BBI01}) as well as the
theme that results which are true in Riemannian spaces of sectional
curvature bounded above are often true --- and often have more
transparent proofs --- in the broader class of \catk\ spaces.
Alternatively,  the reader may consult the short appendix containing
the definitions and basic tools that we use.  We hope that readers
based in comparison geometry will find both the theorems on the
asymptotics of total curvature, and their applications to the capture
problem, of interest;  and readers interested in pursuit-evasion games
will find the power of comparison geometry compelling.


\subsection{Motivation}

The application to pursuit-evasion games requires some
motivation and background. There is a significant literature on
pursuit and evasion games with natural motivations coming from
robotics, control theory, and defense applications
\cite{Guibas,ISS,Vidal}. Such games involve one or more {\em
evaders} in a fixed domain being hunted by one or more {\em
pursuers} who win the game if the appropriate capture criteria
are satisfied. Such criteria may be physical capture (the
pursuers move to where the evaders are located)
\cite{Isaacs,IKK,Parsons} or visual capture (there is a
line-of-sight between a pursuer and an evader)
\cite{Guibas,SY92}. The types of pursuit games are many and
varied: continuous or discrete time, bounded or unbounded
speed, and constrained or unconstrained acceleration, energy
expenditure, strategy, and sensing. For a quick introduction to
the literature on pursuit games, see, e.g., \cite{KR,Guibas}.

The applications in this paper focus on one particular variable in
pursuit games: the geometry and topology of the domain on which the
game is played. We keep all other variables fixed and as simple as
possible. Thus, there will be a single pursuer-evader pair, a single
simple pursuit strategy (`move toward the evader'), and no
constraints on acceleration or related system features.

The vast majority of the known results on pursuit-evasion are
dependent on having Euclidean domains which are two-dimensional
or, if higher-dimensional, then convex. There has of late been a
limited number of results for pursuit games on surfaces of revolution
\cite{HM00}, cones \cite{Mel98}, and round spheres \cite{Kov96}. Our
results are complementary to these, in the sense that we work with
domains of arbitrary dimension, with no constraints on being either
smooth or locally Euclidean.

There are several reasons for wanting to extend the study of
pursuit-evasion games to the most general class of playing fields
possible. The most obvious such application is in the generalization
from 2-d to 3-d, in which the pursuit game is a model for physical
pursuit, as well as in the expansion to nonconvex domains.    For
example, a closed, simply connected domain $\Domain$ with smooth
boundary in $\euc^3$ is \catzero\ if the tangent plane at every
boundary point $p$ contains points arbitrarily close to $p$ that are
not in the interior of $\Domain$. (This is a special case of the
characterization of upper curvature bounds of manifolds with
boundary in \cite{ABB93}.)  More generally, a domain in $\euc^3$ is
\catone\ if it is not too far from convex, that is, its boundary is not
too outwardly curved;  see Theorem \ref{thm:catone-example} below.

However, higher dimensional playing fields can also correspond to
physical problems, via {\em configuration spaces} of physical
systems. Consider the following (fanciful) example. If one
wants to mimic the action of a dancer with a complex robot, one
could attempt a generalized pursuit game in which the playing
field is the configuration space of the dancer's (or robot's)
motions. The dancer's configuration plays the role of the
evader, and the robot's configuration plays the role of the
pursuer. If the robot's goal is to mimic the dancer in real
time with knowledge only of the dancer's instantaneous body
configuration, then this translates into a simple pursuit
problem with one pursuer and one evader. The results of this
paper show that no matter how high the dimension of the
configuration space, the pursuit strategy will be successful if
the configuration space is \catzero. It has been demonstrated
recently that there is a significant class of configuration
spaces in robotics and related fields which do have an
underlying \catzero\ geometry \cite{AG,BHV,GL,GP}, rendering
the cartoon example above a little less unrealistic. In like
vein, work on consensus, rendezvous, and flocking \cite{TJP} is
a form of coordinated pursuit in which the evader is the
consensus or rendezvous state(s).

\section{Total curvature}
\label{sec_TotalCurvature}

For a curve in a \catzero\ space, successively stronger constraints on
the total curvature function control  long-term behavior.

\subsection{Definitions}

\begin{definition}\label{def:tc}
The \emph{total rotation} $\tau_\sigma$ of a polygonal (i.e.,
piecewise-geodesic) curve $\sigma$ is $\sum_i {(\pi - \beta_i)}$,
where the $\beta_i\geq 0$ are the angles at the interior vertices. The
\emph{total curvature} of any curve $\gamma$ is the $\limsup$ of
$\tau_\sigma$ as $\mu_\sigma\to 0$, over all polygonal $\sigma$
inscribed in $\gamma$, where $\mu_\sigma$ is the maximum segment
length of $\sigma$.
\end{definition}

In \catzero\  spaces, monotonicity of total rotation follows from
triangle comparisons: if $\sigma$ is inscribed in a polygonal curve
$\gamma$, then $\tau_\sigma\le\tau_\gamma$ \cite{AB98}. Thus the
total curvature of any curve $\gamma$ in a \catzero\ space is the
supremum of $\tau_\sigma$ over all polygonal $\sigma$ inscribed in
$\gamma$.  Monotonicity of total rotation fails in \catk\ spaces for
$K>0$, but a more subtle argument proves that the total curvature of
an arbitrary curve is the limit of $\tau_{\sigma_n}$ for any sequence
$\sigma_n$ of inscribed polygonal curves  with $\mu_\sigma\to 0$
\cite[see Theorem \ref{thm:tc-limit} below]{ML03}.

In particular, the  total rotation $\tau_\sigma$ of a polygonal curve
coincides with its total curvature. Accordingly: \emph{from now on,
we denote the total curvature of an arbitrary curve $\gamma$ in a
\catk\ space by  $\tau_\gamma$}.

\begin{ex}
If $\gamma$ is a unitspeed curve in $\euc^n$, then $\tau_\gamma$
equals the length  in the unit sphere of the curve $\gamma^{'+}$ of
righthand unit tangent vectors, with jump discontinuities replaced by
great circular arcs \cite{AR}. In particular, if $\gamma$ is smooth in
$\euc^2$, so that $\gamma'(t)=(\cos \theta (t),\sin \theta (t))$, then
$\tau_\gamma = \int \kappa$, where $\kappa = |\gamma''| =
|\theta'|$.
\end{ex}

Curves of finite total curvature in
\catk\ spaces are well-behaved, in the sense that they have
unit-speed parametrizations, which have left and right unit velocity
vectors at every point  \cite{ML03}.

We are interested in how the asymptotic behavior of the \emph{total
curvature function},
$$\tau(t) = \tau_{\rest{\gamma}{[0,t]}},$$
controls the function that measures the maximum distance from its
initial point realized by the curve in a given time period:

\begin{definition}
The \emph{circumradius function} of a curve $\gamma$ is the
real-valued function $c$, where $c(t)$ is the smallest number such
that the path $\rest{\gamma}{[0,t]}$ lies in the ball of radius
$c(t)$ about $\gamma(0)$.
\end{definition}

\subsection{Growth rate of total curvature and circumradius}
The following theorem will be applied in Section \ref{sec_Escape} to
pursuit-and-evasion games, to obtain a necessary condition for the
evader to win, in terms of how far from home the evader wanders
during given time periods.  Theorem \ref{thm:unbounded} generalizes
a theorem of Dekster for Riemannian manifolds \cite{Dek80}.
However, we use Reshetnyak majorization to obtain a simple
argument that moreover holds for any \catzero\ domain.

\begin{theorem}\label{thm:unbounded}
For any curve $\gamma$, parametrized by arclength $t$, in a
\catzero\ space, let $\tau$  and $c$ be its total curvature and
circumradius functions.
\begin{enumerate}
\item[(a)] If  $\liminf_{t\to\infty} \tau(t)/t =0$, then $\gamma$
is unbounded.
\item[(b)] If $\tau\in O(t^a)$ for some $a\in (0,1)$, then
$c\in\Omega(t^{1-a})$.
\end{enumerate}
\end{theorem}

\begin{ex}
Consider the spiral $\gamma(u)=(u\cos 2\pi u, u\sin 2\pi u)$ in
$\euc^2$.  The total curvature function is linear in $u$, as is the
circumradius, while the arclength $t$ grows quadratically. This is
case (b) for $a = \frac{1}{2}$, with $t/c(t)^2$ bounded.
\end{ex}

\begin{proof}[Proof of Theorem \ref{thm:unbounded}]
For part (a), we may suppose by approximation that any fixed initial
segment $\rest{\gamma}{[0,t]}$ is polygonal. Subdivide $[0,t]$ into
at most $\frac{\tau(t)}{\pi/2} + 1$ subintervals so that the
restriction $\gamma_i$ of $\gamma$ to each subinterval has total
curvature at most $\pi/2$. (If any angles are less than $\pi/2$, we
first refine the polygon by cutting across each such angle with a
short segment to obtain two angles of at least $\pi/2$. )  Let
$\rho_i$ be the closed polygon consisting of $\gamma_i$ and its
chord $\sigma_i$. By Reshetnyak majorization, there is a closed
convex curve $\widetilde{\rho}_i$ in $\euc^2$ that majorizes
$\rho_i$. Since a majorizing map preserves geodesics and does not
increase angles, $\widetilde{\rho}_i$ is a closed polygon with the
same sidelengths as $\rho_i$, consisting of a polygonal curve
$\widetilde{\gamma}_i$ and its chord $\widetilde{\sigma}_i$, where
the total curvature of $\widetilde{\gamma}_i$ is at most $\pi/2$.

Since $\widetilde{\gamma}_i$ is a convex curve in $\euc^2$ having
total curvature at most $\pi/2$, the ratio of its length to that of
its chord is at most $\sqrt{2}$ (the ratio of two sides of an
isosceles right triangle to its hypotenuse).  Therefore
\begin{equation}\label{eq:arc-estimate}
t \le \left(\frac{\tau(t)}{\pi/2}+1\right)\sqrt2 \sup |\sigma_i|,
\end{equation}
so
\begin{equation*}
\frac{\tau(t)}{t} \ge \frac{\pi}{2}\left(\frac{1}{\sqrt2\sup
|\sigma_i|} - \frac{1}{t}\right) .
\end{equation*}
But if $\gamma$ is bounded, so that $\sup |\sigma_i|<\infty$, it
follows that $\tau(t)/t$ is bounded away from $0$ for $t$
sufficiently large. This proves Part (a).

For Part (b), if one substitutes $\tau(t) \le At^a$ and
$\sup|\sigma_i| \le 2c(t)$ in (\ref{eq:arc-estimate}), it is
immediate that $t/c(t)^{1/(1-a)}$  is bounded.
\end{proof}

\subsection{Finite total curvature and asymptotic rays}

Total curvature also controls how close an infinite curve of finite
total curvature must be to a geodesic ray.  In the Riemannian
setting, the conclusions of the following theorem were obtained
by Langevin and Sifre \cite{LS} under stronger hypotheses.  That is,
they assume the pointwise curvature $\kappa(t)$ of a smooth curve
$\gamma$ satisfies $\kappa(t)\in O(t^{-1-\epsilon})$ in Part (a),
and the same with $t^{2+\epsilon}$ in Part (b). Here again, we give
a simple argument using \catzero\ techniques.

A curve $\gamma$ is said to be \emph{asymptotic} to a geodesic ray
$\sigma$ if $d(\gamma(t), \sigma)$ is bounded.  Now we show  that a
curve of finite total curvature $\tau_\gamma$ always has sublinear
distance to some geodesic ray, to which it is asymptotic if the total
curvature function approaches its limit $\tau_\gamma$ sufficiently
rapidly.

\begin{theorem}\label{thm:asymptotic}
Let $\gamma$ be a curve, parametrized by arclength $t$, in a
\catzero\ space. Suppose $\gamma$ has finite total curvature $
\tau_\gamma = \lim_{t\to\infty} \tau(t)$.
\begin{enumerate}
\item[(a)]
    Through any point $p$, there is geodesic ray $\sigma$ such
    that $d(\gamma(t),\sigma)\in o(t)$.
\item[(b)] The circumradius function satisfies $c\in\Omega(t)$.
\item[(c)]
    If $\int_0^\infty(\tau_\gamma - \tau(t)) dt < \infty$, then
    $\gamma$ and $\sigma$ are asymptotic.
\end{enumerate}
\end{theorem}

\begin{proof}
Again, we may assume $\gamma$ is polygonal. Choose an increasing
sequence $t_i \to \infty, i \ge 0, t_0 = 0$. Let $\sigma_i$ be the
geodesic joining $\gamma(0)$ and $\gamma(t_i)$. For $i \ge 1$, let
$\rho_i$ be the closed polygon made up of $\sigma_{i-1}, \gamma_i =
\rest{\gamma}{[t_{i-1},t_i]}$, and $\sigma_i$. We denote by
$\widetilde{ \rho}_i$ a convex polygon in $\euc^2$ that majorizes
$\rho_i$, and by $\widetilde{\gamma}_i$ its subarc corresponding to
$\gamma_i$. We suppose the $\widetilde{\rho}_i$ are arranged in a
counterclockwise ``fan'', so that $\widetilde{\rho}_{i+1}$
intersects $\widetilde{\rho}_i$ along the straight line segment
$\widetilde{\sigma}_i$ in each that corresponds to $\sigma_i$, and
the points corresponding to $\gamma(0)$ coincide at the centerpoint
$\widetilde{O}$.

Consider the polygonal curve $\widetilde{\gamma} : [0,\infty) \to
\euc^2$, parametrized by arclength, whose image is the union of the
convex curves $\widetilde{\gamma}_i$. The restriction of the
majorizing map of $\widetilde{\rho}_i$ to $\widetilde{\gamma}_i$
maps onto $\gamma_i$ and does not increase angles at vertices. Thus
it does not decrease total curvature. Moreover, the interior angles
of $\widetilde{\rho}_i$ at $\widetilde{\gamma}(t_{i-1})$ and
$\widetilde{\gamma}(t_i)$ are at least those of $\rho_i$ at the
corresponding points. Let $\widetilde{\beta}_i$ be the sum of the
angles of $\widetilde{\rho}_i$ and $\widetilde{\rho}_{i+1}$ at
$\widetilde{\gamma}(t_i)$. By the triangle inequality for angles in
$M$, $\widetilde{\beta}_i$ is at least the angle between the left
and right directions of $\gamma$ at $t_i$. Thus the angle between
the initial and final tangents of
$\rest{\widetilde{\gamma}}{[t_m,t_n]}$ in $\euc^2$, namely, the sum
of the $(\pi - \widetilde{\beta}_i)$, which could be negative,  and
the positive  total curvatures of the convex polygonal curves
$\widetilde{\gamma}_i$ for $m+1 \le i \le n$, is no more than the
total curvature of $\rest{\gamma}{[t_m,t_n]}$. Now take $t_m = 0$.
Since $\widetilde{\gamma}(0)=\widetilde{O}$, the total angle at
$\widetilde{O}$ of the first $n$ sectors of the fan is no more than
the angle between the initial and final tangents of
$\rest{\widetilde{\gamma}}{[0,t_n]}$, and hence no more than
$\tau_\gamma$. Therefore the vertex angles of the
$\widetilde{\rho}_i$ are summable, and the angle between
$\widetilde{\gamma}$ and the ray from $\widetilde{O}$ through
$\widetilde{\gamma}(t)$ converges to $0$.

Let  $\widetilde{r}(t)=d(\widetilde{O},\widetilde{\gamma}(t))$.  By
the First Variation Formula (\ref{first-variation}), the one-sided derivatives
$d\widetilde{r}/dt$ converge to $1$. Thus for any $A < 1$ we have
$\widetilde{r}(t)$ increasing and $\widetilde{r}(t) > At$ for $t$
sufficiently large. Furthermore, for each choice of sequence $t_i \to
\infty$ our construction produces a function $\widetilde{r}$ satisfying
$\widetilde{r}(t_i) = r(t_i)$, where $r(t)=d(O,\gamma (t))$. It follows
that $r(t)$ also eventually increases and $r(t)> At$ for $t$
sufficiently large.

Since the directions of the line segments $\widetilde{\sigma}_i$ at
their basepoint $\widetilde{O}$ converge and $|\widetilde{\sigma}_i|
\to \infty$, the $\widetilde{\sigma}_i$ converge to a Euclidean ray
$\widetilde{\sigma}$ from $\widetilde{O}$. Let $s$ be the arclength
parameter on $\widetilde{\sigma}$. Since the angle at which a ray
strikes $\widetilde{\gamma}$ converges to $0$, it follows that for
$t$ sufficiently large, $\widetilde{\gamma}$ is the graph of a
height function of order $o(s)$ over $\widetilde{\gamma}$. Hence
$d(\widetilde{\gamma}(t), \widetilde{\sigma}) \in o(s(t)) = o(t)$.

Now since each $\widetilde{\rho}_i$ majorizes $\rho_i$, the
intersections of the geodesics $\sigma_i$ with any ball in $M$ about
$\gamma(0)$ converge to a geodesic. Therefore the $\sigma_i$
converge to a geodesic ray $\sigma$ in $M$. Furthermore,
$d(\widetilde{\gamma}(t), \widetilde{\sigma})$ is realized by a line
segment through the fan, infinitely partitioned by its intersections
with a truncated sequence of the $\widetilde{\sigma}_i$. Since
$\widetilde{\rho}_i$ majorizes $\rho_i$, $\gamma(t)$ is joined to
$\sigma$ by a path of no greater length. Therefore
$d(\gamma(t),\sigma) \le d(\widetilde{\gamma}(t),
\widetilde{\sigma}) \in o(t)$, as claimed in Part (a).

Part (b) follows from (a) trivially, given the linear circumradius
of geodesic rays.

Since $d(\gamma(t),\sigma) \le d(\widetilde{\gamma}(t),
\widetilde{\sigma})$, it suffices for Part (c) to show that the
latter is bounded.  The length of the projection of
$\widetilde{\gamma}$ in $\euc^2$ to  a line normal to
$\widetilde{\sigma}$ is obtained by integrating $\sin\beta(t)$,
where $\beta$ the angle between the righthand tangent of
$\widetilde{\gamma}$ and the direction of $\widetilde{\sigma}$.
Since
\[
\sin\beta(t) \le \beta(t) \le
\tau_{\rest{\widetilde{\gamma}}{[t,\infty]}} \le
\tau_{\rest{\gamma}{[t,\infty]}} \le \tau_\gamma - \tau(t) ,
\]
Part (c) follows.
\end{proof}

\begin{remark}
By alternately gluing Euclidean and hyperbolic bands bounded by pairs
of asymptotic geodesics, one can construct a \catzero\ space in which
the curves of Part (c) need not have {\em strict} asymptotes (the
distance to which approaches $0$ rather than merely being bounded).
This construction is carried out in \cite{LS}, although the \catzero\
nature of the resulting glued space is not mentioned.
\end{remark}

\section{Simple pursuit on \catzero\ domains}
\label{sec_Pursuit}

The following rules define a basic discrete-time equal-speed pursuit
game. (Continuous-time pursuit will be discussed in Section
\ref{sec_continuous}.)  Let $(X,d)$ denote a geodesic metric space
(representing the domain on which the game is played). There is a
single pursuer $P$ and a single evader $E$ starting at locations $P_0$
and $E_0$ respectively. At the $i$-th step, the evader moves from
$E_{i-1}$ to $E_{i}$, a point within distance $D$ chosen by the
evader. The pursuer moves to $P_{i}$, the point along a geodesic
from $P_{i-1}$ to $E_{i-1}$ at distance $D$ from $P_{i-1}$. The
moves are illustrated in Figure \ref{fig_DTCap}. Four points, $P_{i}$,
$E_{i}$,  $E_{i+1}$  and $P_{i+1}$, form a degenerate geodesic
quadrangle with side lengths $L_{i}$, $d(E_i, E_{i+1}) \le D$,
$L_{i+1}$, and $D$, where $L_{i}=d(P_{i},E_{i})$ for each nonnegative
integer $i$.

\begin{figure}[hbt]
\begin{center}
\psfragscanon
\psfrag{e}[][]{$E_i$}
\psfrag{E}[][]{$E_{i+1}$}
\psfrag{p}[][]{$P_{i}$}
\psfrag{P}[][]{$P_{i+1}$}
\psfrag{a}[][]{$\alpha_i$}
\psfrag{g}[][]{$\varphi_i$}
\psfrag{d}[][]{$\delta_i$}
\psfrag{b}[][]{$\beta_i$}
\includegraphics[width=3.5in]{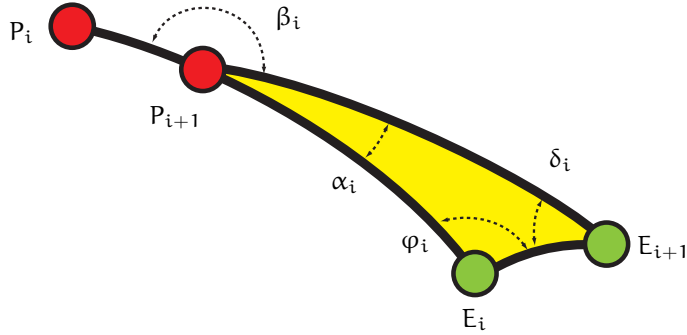}
\caption{A degenerate quandrangle arising from a discrete-time capture
problem.}
\label{fig_DTCap}
\end{center}
\end{figure}

This type of motion, in which the pursuer moves in an unconstrained
fashion in the direction of the evader, is called {\em simple pursuit}.
Given $P_0$ and the sequence $\{E_{i}\}$, we say that $P$ wins if
$d(P_i,E_i)\leq D$ for some $i$; otherwise, $E$ wins. We note that
this instantaneous, memoryless strategy for pursuit is not necessarily
the pursuer's optimal strategy \cite{Sgall,KR} --- merely the simplest.

By the triangle inequality,
\begin{equation}\label{eq:monotone}
L_{i+1}\leq d(P_{i+1},E_i)+D=L_{i}.
\end{equation}
Thus $\lim_{i\to\infty}L_{i}=L$
exists, and the evader wins if and only if this limit is greater than
$D$.  Moreover,
\begin{equation}\label{eq:alpha}
\pi - \beta_i\le\alpha_i\le\widetilde{\alpha}_i,
\end{equation}
where $\alpha_{i}$ is the angle between the geodesics joining
$P_{i+1}$ to $E_i$ and $E_{i+1}$, and $\widetilde{\alpha}_i$ is the
angle corresponding to $\alpha_i$ in the Euclidean triangle with the
same sidelengths as $\triangle E_{i}P_{i+1}E_{i+1}$ (compare
Definition \ref{def:tc}).  The first inequality in (\ref{eq:alpha}) is by
the triangle inequality  for the \emph{angle distance} between the
directions of geodesic segments with a common origin. These
observations highlight the naturality of the \catzero\ definition in the
context of pursuit problems.

A discrete-time {\em pursuit curve} $P(t)$ is obtained by joining the
$P_i$ by geodesic segments, where $t$ has speed $1$.  Thus $P_i =
P(iD)$, where $D$ is the step size.  The discrete-time {\em evader
curve} $E(t)$ is defined similarly;  however, since the evader's step
sizes are assumed $\le D$, on each geodesic segment the parameter
$t$ has $\triangle t=D$ and constant speed $\le 1$.

The following simple result is well known for convex Euclidean
domains. Theorem \ref{thm:unbounded} provides us with the
immediate extension to \catzero\ domains:

\begin{theorem}\label{discrete_compact}
For  discrete-time simple pursuit on a complete \catzero\ domain, the
domain is compact if and only if the pursuer always wins.

\end{theorem}

\proof If the evader wins, then by (\ref{eq:monotone}) we have
$L_{i+1}\to L$ and $d(P_{i+1},E_i)+D\to L$.  Therefore the angle
$\alpha_{i}$ vanishes in the limit, because the same is true for
$\widetilde{\alpha}_i$. By (\ref{eq:alpha}), the total curvature of the
$P$ curve is sublinear.  This curve is unbounded via Theorem
\ref{thm:unbounded}, so the domain is noncompact.

Conversely, a noncompact  \catzero\ domain contains an infinite
geodesic ray, along which the evader and pursuer may move with
constant separation and hence without capture. \qed

We now relate the total curvature functions  $\tau^E(t)$ of the
evader curve $E(t)$ and $\tau^P(t)$ of the pursuer curve $P(t)$.  On
the one hand, the evader may accumulate large total curvature by
zigzagging, without much affecting the pursuer's total curvature. On
the other hand, the pursuer's total curvature may exceed the
evader's:   if  E runs along a geodesic ray and  P  does not start on the
ray, then the evader's total curvature is $0$ and the pursuer's is
positive. The following result makes these observations precise.  The
proof uses only angle comparisons, and is added evidence that simple
pursuit has an affinity for the CAT(0) setting.

\begin{theorem}\label{thm:evader-tc}
For discrete-time simple pursuit on a  \catzero\ domain, the total
curvature functions of evader and pursuer satisfy
$$\tau^P(t)\le \tau^E(t)+\pi.$$
\end{theorem}

\begin{proof}
Set $(\tau^P)^{n+1}=\tau^P\bigl((n+1)D\bigr)$, the total curvature of
the pursuit curve to $P_{n+1}$, and similarly for $(\tau^E)^{n+1}$.
Label the internal angles of $\triangle E_{i}P_{i+1}E_{i+1}$ by
$\alpha_i, \varphi_i,\delta_i$ as indicated in Figure \ref{fig_DTCap}.
Then

$$(\tau^P)^{n+1}= \sum_0^{n-1} {(\pi - \beta_i)}
\le \sum_0^{n-1}\alpha_i \le \sum_0^{n-1} {(\pi - \varphi_i-\delta_i)},$$
since $\alpha_i+\varphi_i+\delta_i\le\pi$ by the \catzero\ condition.

On the other hand, letting $\theta_i$ be the interior angle of the
evader curve at $E_i$, we have
$$(\tau^E)^n=\sum_1^{n-1}{(\pi-\theta_i)}\ge\sum_1^{n-1} {(\pi - \delta_{i-1}-\varphi_i)},$$
since $\theta_i\le \delta_{i-1}+\varphi_i $ by the triangle inequality
for angle distance between directions at $E_i$. Therefore
$$(\tau^P)^{n+1}-(\tau^E)^n\le \pi -\varphi_0-\delta_{n-1}. $$
\end{proof}

\section{Escape}
\label{sec_Escape}

On a noncompact domain, the relevant question is whether the
evader can escape when the pursuer adopts the simple pursuit-curve
strategy, and, if so, what conditions lead to escape. Here we show
that the pursuer still always wins if the circumradius of the evader
does not grow fast enough, or, equivalently via Theorem
\ref{thm:unbounded}, if the pursuit path is forced to curve too much.
The proof of this necessary condition for escape uses an estimate on
total curvature of pursuit curves that will be proved in Theorem
\ref{thm:TCbound} of the next section.

\begin{theorem}\label{thm:noncompactcap}
Suppose the evader wins a discrete-time simple pursuit on a \catzero\
domain $\Domain$. Then the total curvature $\tau(t)$ of the pursuit
curve from $P_0$ to $P(t)$ is $O(t^{\frac{1}{2}})$.
\end{theorem}

\begin{proof}[Proof, assuming Theorem \ref{thm:TCbound}]
We invoke the facts that a \catzero\ domain is also a \catone\
domain, and a rescaling of a \catzero\ domain is again a \catzero\
domain.

Theorem \ref{thm:TCbound} states that for simple pursuit on a
\catone\ domain, if the evader wins and the initial distance $L_0 =
d(P_0,E_0)$ is less than $\pi$, then the total curvature $\tau(t)$ of
the pursuit curve from $P_0$ to $P(t)$ is $O(t^{\frac{1}{2}})$.
Therefore we rescale the metric of $\Domain$ by $\pi/L_0$. Since
angles are invariant under rescaling, then by Definition \ref{def:tc},
the total curvature of a given segment of the pursuit curve is also
invariant under rescaling.
\end{proof}

\begin{cor}
\label{cor:escapecircumradius}
Let $c$ denote the circumradius function of an evader's path on a \catzero\
domain. Then $c\in\Omega(t^\frac{1}{2})$ is a necessary condition
for the evader to win in a discrete-time simple pursuit game.
\end{cor}

\begin{proof}
Combine Theorem \ref{thm:noncompactcap} with Part (b) of Theorem
\ref{thm:unbounded}, using $a=\frac{1}{2}$.
\end{proof}

\section{Domains with positive curvature bounds}
\label{sec_catk}

In applications, spaces with positive curvature are not merely possible
but prevalent. In this section, we demonstrate that controlled
amounts of positive curvature are admissible, as long as we control
initial distances between the pursuer and evader.

First we provide a large class of nonconvex examples of \catk\
domains in $\R^n$:

\begin{theorem}\label{thm:catone-example}
A closed domain $\Domain$ in $\R^n$ with smooth boundary
$\partial\Domain$, where $\Domain$ carries its intrinsic metric, is a
\catk\  space for $K>0$ if it is supported at every
$p\in\partial\Domain$ by a sphere of radius $1/\sqrt{K}$, that is,
every point at distance $\le 1/\sqrt{K}$ from $\Domain$ is the center
of a closed ball that meets $\Domain$ in a single point.
\end{theorem}

\begin{proof}
By \cite[Theorem 3]{ABB87}, the hypothesis of supporting balls
implies that geodesics of $\Domain$ of length $<\pi/\sqrt{K}$ are
uniquely (and hence continuously) determined by their endpoints. By
the Alexandrov patchwork construction (see \cite[p.199]{BH99}), it
follows that $\Domain$ is a  \catk\  space.
\end{proof}

We now study the asymptotic behavior of total curvature of pursuit
curves in \catk\ spaces for $K>0$.  Rescaling the metric by the factor
$1/\sqrt{K}$, we may assume $K=1$.

Just as in the \catzero\ case, the triangle inequality  for $\triangle
P_{i+1}E_{i+1}P_i$ in a \catone\ space easily implies that the
distances $L_i = d(P_i,E_i)$ are monotonically non-increasing: see
Figure \ref{fig_DTCap}. The condition for equality from one step to
the next is that $P_{i+1}$ and $E_{i}$ are on a geodesic segment
$P_iE_{i+1}$, and hence the angles $\angle P_iP_{i+1}P_{i+2}$ and
$\angle P_{i+1}E_iE_{i+1}$ are both $\pi$.

\begin{theorem}
\label{thm:TCbound} On a \catone\ domain, suppose the evader wins
a discrete-time simple pursuit with initial distance $L_0 =
d(P_0,E_0)<\pi$. Then the total curvature $\tau(t)$ of the pursuit
curve from $P(0)$ to $P(t)$ is $O(t^{\frac{1}{2}})$.
\end{theorem}
\begin{proof}
Since the distances $L_i=d(P_i,E_i)$ are monotonically
nonincreasing, all $L_i < \pi$ and the triangles $\triangle
P_{i+1}E_iE_{i+1}$ have perimeters $<2\pi$. Thus, they have model
triangles in the unit sphere.  Let the angle corresponding to $\alpha_i$
be $\tilde\alpha_i$, so that $\alpha_i \le
\tilde\alpha_i$ by the \catone\ condition.  Hence
$$\tau(t) \le \sum_{i=0}^{n-1}\alpha_i\le \sum_{i=0}^{n-1}\tilde\alpha_i.$$
Since the evader wins, we have $L_\infty >D$.  We set
$\Delta_i = L_i - L_{i+1}$.

Apply the spherical law of cosines to the model triangles:
\begin{align}
\cos D \le \cos d(E_i,E_{i+1})
&= \cos L_{i+1} \cos(L_i - D)
+ \sin L_{i+1} \sin(L_i - D) \cos\tilde\alpha_i \notag\\
&= \cos L_{i+1} \cos(L_i - D) +\sin L_{i+1} \sin(L_i - D)\notag \\
&\quad - \sin L_{i+1} \sin(L_i - D)(1 - \cos\tilde\alpha_i)\notag \\
&= \cos(D - \Delta_i)
- \sin L_{i+1} \sin(L_i - D) (1 - \cos\tilde\alpha_i)\notag \\
&\le \cos D + \Delta_i \sin D - B_\infty \tilde\alpha_i^2/5\label{eq:catone},
\end{align}
where $B_\infty = \inf_i\{\sin L_{i+1} \sin(L_i - D)\}$. The last
inequality depends on two elementary inequalities. To verify that
$\cos(D-\Delta_i) \le \cos D + \Delta_i\sin D$, apply the Mean Value
Theorem for $\cos x$ on the interval $D - \Delta_i \le x \le D$. To
verify that $1 - \cos \alpha_i \ge  \alpha_i^2/5$, use the identity $1 -
\cos x = 2\sin^2(x/2)$. Then since $\sin(x/2)$ is concave on the
interval $0 \le x \le \pi$, its graph is above the chord: $\sin(x/2) \ge
x/\pi$. Since $\pi^2 < 10$, we have $2/\pi^2 > 1/5$.

Inequality (\ref{eq:catone}) yields an inequality for $\alpha_i^2$:
$$\alpha_i^2 \le 5\Delta_i\sin D/B_\infty.$$
The following calculation completes the estimate for total
curvature. It uses the Cauchy inequality
$$\tau(t) = \tau(nD) = \sum_{i = 0}^{n-2} \alpha_i
\le \sqrt{(n-1)\sum_{i = 0}^{n-2}\tilde\alpha_i^2},$$
and the telescoping sum
$$\sum_{i = 0}^{n-2}\Delta_i = L_0 - L_{n-1} \le L_0 - L_\infty.$$
Hence
$$\tau(t)
\le \sqrt{(n-1) 5\sin D(L_0 - L_\infty)/B_\infty}
\le C\sqrt{nD} = C\sqrt t,$$
where
$$C =  \sqrt{5(L_0 - L_\infty)/B_\infty}.$$
\end{proof}

\begin{ex} \label{ex:antipodal}
The hypothesis $L_0=d(P_0,E_0)<\pi$ in Theorem \ref{thm:TCbound}
is necessary.  For example, let $\Domain$ be the complement in
$\R^n$ of one or more disjoint open balls of radius $>1$.  Then
$\Domain$ is a \catone\ domain by Theorem
\ref{thm:catone-example}.  Start with $P$ and $E$ antipodal on the
boundary of one of the balls.  $P$ moves toward $E$ a distance
$D<\pi$ around the boundary, while $E$ moves to the antipodal point;
from then on, $P$ and $E$ can oscillate between the two antipodal
pairs.  The total curvature of $P$ increases by $\pi$ at each step, so
grows linearly rather than $O(t^{\frac{1}{2}})$.
\end{ex}

\begin{ex}\label{ex:periodic}
In \catk\ domains for $K>0$, there is no circumradius estimate
analogous to that of Corollary \ref{cor:escapecircumradius}.  As a
simple example of escape with bounded circumradius, consider the
domain $\Domain$ of Remark \ref{ex:antipodal}.  Let $P$ and $E$
travel at constant distance $<\pi$ apart around a  local geodesic of
$\Domain$ which is a great circle in one of the ball boundaries.
\end{ex}

Bounded escape in the \catone\ setting always exhibits some aspects
of Example \ref{ex:periodic}:

\begin{prop}\label{prop:infinite-geodesic}
Let $\Domain$ be a compact \catone\ space.   Suppose the evader
wins a discrete-time simple pursuit with initial distance
$L_0=d(P_0,E_0)<\pi$.  Then there is a bilaterally infinite local
geodesic in $\Domain$, any finite segment of which is the limit of
segments of the pursuit curve.
 \end{prop}

The proof is immediate from Theorem \ref{thm:TCbound} and the
following lemma:

\begin{lemma}\label{lem:periodic}
Let $\Domain$ be a compact \catone\ space. Suppose the total
curvature function $\tau(t)$ of a curve
$\gamma:[0,\infty)\to\Domain$ has sublinear growth.   Then there is
a bilaterally infinite local geodesic in $\Domain$, any finite segment
of which is the limit of segments of $\gamma$.
\end{lemma}

We draw on the following work of Maneesawarng et al:

\begin{theorem}\label{thm:tc-limit}
In a \catone\ space, total curvature has the following properties.
\begin{enumerate}

\item\label{tc-semicont} {\em Semi-continuity \cite{KM07}:}  if a
    sequence of polygonal curves  with total curvatures $\tau_m$
    converges uniformly on the same parameter interval to a curve
    $\gamma$, then $\tau_\gamma \le \liminf\tau_m.$

\item\label{tc-continuity} {\em Continuity under inscription
    \cite{M99,ML03}:}  if a sequence of  polygonal curves  with total
    curvatures $\tau_m$ is inscribed in a curve $\gamma$ so the
    maximum diameters $d_m$ of the broken segments of
    $\gamma$ approach $0$, then $\tau_\gamma = \lim\tau_m.$

\item\label{tc-length} {\em Length estimate \cite{M99,ML03}:} Let
    $\gamma$ be a curve from $p$ to $q$, with $\tau_\gamma +
    d(p,q)< \pi$.  Then $|\gamma|$ is at most the length of an
    isosceles once-broken geodesic in the unit sphere having the
    same total curvature and endpoint separation.
\end{enumerate}
\end{theorem}

 \begin{proof}[Proof of Lemma \ref{lem:periodic}]
We claim that for any $\epsilon>0$ and $T>0$, there is a sequence
$t_{n}\to\infty$ such that the total curvature $\tau(t)$ of $\gamma$
satisfies $\tau(t_n+T)-\tau(t_n-T)<\epsilon$.  Otherwise, any
increasing sequence of such $t_n$ would have a finite supremum,
after which the growth of $\tau$ would be linear, contradicting the
hypothesis. Choosing sequences $\epsilon_i\to0$ and $T_i=2i$,  and
selecting one $t_i$ for each $i$, yields a sequence $t_i$ satisfying
\begin{equation}\label{vanishing-tc}
\tau(t_i+i)-\tau(t_i-i)\to 0,\ t_i\to\infty.
\end{equation}

Let $\gamma_{ik}$ be the the restriction of $\gamma$ to
$[t_{i}-k,t_{i}+k]$, $i\ge k$, reparametrized by arclength on $[-k,k]$.
Writing $\tau_{ik}=\tau_{\gamma_{ik}}$, we have
$\lim_{i\to\infty}\tau_{ik}=0$ by  (\ref{vanishing-tc}).

By compactness of $\Domain$, the $\gamma_{i1}$ have  a
subsequence $\tilde{\gamma}_{i1}$ that converges to a curve
$\rho_1$. By Theorem \ref{thm:tc-limit} (\ref{tc-semicont}),
$\tau(\rho_1)=0$ and so $\rho_1$ is a local geodesic.  Theorem
\ref{thm:tc-limit} (\ref{tc-length}) implies $|\gamma_{i1}|\to
|\rho_1|$, and so $\rho_1$  has length $2$.

By construction, the $\tilde{\gamma}_{i1}$ for $i\ge 2$ extend to
subsegments  of the pursuit curve of length $4$ that form a
subsequence of the sequence $\gamma_{i2}$.  From this subsequence
we may extract a further subsequence $\tilde{\gamma}_{i2}$ that
converges to a curve $\rho_2$ of length $4$.  As before,  $\rho_2$ is
a local geodesic. Since the restrictions to $[-1,1]$ of the
$\tilde{\gamma}_{i2}$ converge to $\rho_1$, then $\rho_2$ extends
$\rho_1$.

In this manner, we obtain local geodesics $\rho_k$ of length $2k$ for
any $k$, each an extension of the preceding one;  and hence obtain a
local geodesic $\rho$ which by construction has the desired property.
\end{proof}

 \begin{remark}\label{rem:periodic}
There may be no \emph{periodic} local geodesic $\rho$ with the
property described in Proposition \ref{prop:infinite-geodesic} or
Lemma \ref{lem:periodic}.  To see this, consider the Thue-Morse
infinite binary word, which we write as a sequence of the integers
$1$ and $2$. Let  $\Domain$ be the \catone\ domain given by the
complement in $\R^2$ of two disjoint open disks of radius $1$.  Let
$\gamma:[0,\infty)\to\Domain$ be a local geodesic that winds around
the two boundary circles  $\partial\Domain_1$ and
$\partial\Domain_2$, according to the pattern dictated by
Thue-Morse word.  That is, the appearance of an integer $i\in\{1,2\}$
indicates that $\gamma$ makes positively oriented contact with
$\partial\Domain_i$, and a subword consisting of $k>1$ repeats of the
integer $i$ indicates that $\gamma$ also consecutively performs
$k-1$ complete, positively oriented circuits around
$\partial\Domain_i$.  Our claim is immediate from the fact that no
subword of the Thue-Morse word repeats three times in a row
\cite[Theorem 1.8.1]{AS}. \end{remark}

\section{Continuous-time pursuit}
\label{sec_continuous}


In the continuous version of simple pursuit, $E$ moves along a
rectifiable curve $E(t)$ parametrized  with speed $\le 1$.  It is
assumed that $P(t)$ moves at constant unit speed, and for each $t$,
the right-handed velocity vector $P'(t)$ exists and points along a
geodesic from $P(t)$ to $E(t)$. Thus a continuous pursuit curve is a
time-dependent gradient curve for the distance function from a
moving point $E(t)$.

In the \catk\ setting, we assume that the initial separation satisfies
$L(0)<\pi/\sqrt{K}$. As follows from the First Variation Formula
(\ref{first-variation}), $L(t)$ is non-increasing, where
$L(t)=d(P(t),E(t))$. Hence the geodesic from $P(t)$ to $E(t)$ is
unique. The evader wins if and only if $\lim_{t\to\infty}L(t)>0$.

Suppose we are given a rectifiable curve $E(t)$, with $t \ge 0$ and
speed $\le 1$, and an initial pursuer position $P(0)$ and positive step
size $D$. The \emph{discrete-time pursuit game $P_{D,i}$ generated
by the data $\{E(t), P(0), D\}$} has evader sequence $E_{D,i} = E(iD)$
and initial pursuit point $P_{D,0} = P(0)$. As in section
\ref{sec_Pursuit}, there are corresponding broken geodesic pursuit and
evader curves $P_D(t)$ and $E_D(t)$.   These discrete-time curves do
not form a continuous-time simple pursuit game unless the evader
curve is a geodesic with $P(0)$ on a left-end geodesic extension. We
denote the separation at time $t$ by $L_D(t) = d(P_D(t), E_D(t))$.

Jun has provided a foundation for the theory of continuous simple
pursuit and its approximation by discrete simple pursuit, including
existence, uniqueness and curvature properties of continuous pursuit
curves.  In particular we use the following theorem:

\begin{theorem}[\cite{Jun}]
\label{thm:jun}
In a \catk\ space, let $E(t), t \ge 0$, be a rectifiable curve  with
speed $\le 1$, and  $P(0)$ be an initial pursuer position with initial
separation $L(0)=d(P(0),E(0))<\pi/\sqrt{K}$. Consider the
corresponding discrete-time pursuit games $P_{D(m),i}$ with step
sizes $D(m) = 2^{-m}$. Then the sequence of discrete-time pursuit
curves $P_{D(m)}(t)$ for $E_{D(m)}(t)$ converges to a continuous
unit-speed pursuit curve $P(t)$ for $E(t)$, uniformly on any initial arc
$t \le T$.  Moreover,  $P(t)$ is the unique pursuit curve with initial
position $P(0)$.
\end{theorem}

An immediate consequence of Theorem \ref{thm:jun} is that
if the continuous evader wins, then eventually the generated discrete
evaders also win. Indeed, if $m$ is sufficiently large then
\begin{equation}\label{eq:separation}
L_{D(m)}(t) > L(t)/2 > D(m).
\end{equation}
Therefore the continuous versions of Theorem \ref{discrete_compact}
on compact domains, and Corollary \ref{cor:escapecircumradius} on
escape circumradius functions $c$, follow directly from these discrete
theorems and   Theorem \ref{thm:jun}.

Moreover, Jun showed that our estimates from the proof of Theorem
\ref{thm:TCbound} can be used to prove the continuous version of
that theorem.   We do not have to assume that the evader wins in
order to get a bound on total curvature of an initial arc $t \le T = nD$
of the pursuer, only that $L_{n-2} > D$, i. e., the evader has not been
caught yet. Then, evidently, we can replace $C$ by the
time-dependent multiplier
\begin{equation}\label{eq:C_D(T)}
C_D(T)=  \sqrt{\frac{5(L_0 - L_{n-1})}{B_D(T)}},
\end{equation}
where $B_D(T) =  \min\{\sin L_{i+1} \sin(L_i - D): i < n-1\}$.

\begin{theorem}[\cite{Jun}]\label{thm:sqrtestimate}
On a \catone\ domain, suppose the evader wins a continuous simple
pursuit with initial distance $L_0 = d(P_0,E_0)<\pi$. Then the total
curvature $\tau(t)$ of the pursuit curve from $P(0)$ to $P(t)$ is
$O(t^{\frac{1}{2}})$.

\end{theorem}
\begin{proof}
By Theorem \ref{thm:jun} and the inequalities (\ref{eq:separation})
we know that for $m$ sufficiently large, $E_{D(m)}(T)$ is not caught.
Then the limit of the bound on $\tau_{D(m)}(T)$ given by
(\ref{eq:C_D(T)}) is $$C(T)\sqrt{T} = \sqrt{\frac{5(L(0) -
L(T))}{B(T)}}\sqrt{T},$$ where $B(T) = \min\{\sin^2L(0), \sin^2L(T)\}$.
By Theorem \ref{thm:tc-limit} (\ref{tc-semicont}), $$\tau(T) \le
\liminf \tau_{D(m)}(T)\le \lim C_{D(m)}(T) \sqrt T= C(T) \sqrt T.$$
Since the evader wins, then $\tau(t)\le C\sqrt t$ where  $C =
\lim_{T\to\infty} C(T)>0$.
\end{proof}

The continuous case of Theorem \ref{thm:evader-tc} may also be
reduced to the discrete case:

\begin{theorem}
For continuous simple pursuit on a  \catzero\ domain, the total
curvature functions of evader and pursuer satisfy
$$\tau^P(t)\le \tau^E(t)+\pi.$$
\end{theorem}
\proof  The result is immediate from the following chain of inequalities.
$$\tau^P(t) \le \liminf\tau^P_m(t) \le \liminf\tau^E_m(t) + \pi
= \limsup\tau^E_m(t) + \pi = \tau^E(t) + \pi.$$ The first inequality in
the chain is by Theorem \ref{thm:tc-limit} (\ref{tc-semicont}); the
second is by Theorem \ref{thm:evader-tc}; the next (equality) is by
Theorem \ref{thm:tc-limit} (\ref{tc-continuity}), which applies
because the generated discrete evaders $E_m$ are inscribed in $E$;
finally the last is the definition of total curvature. \qed

Proposition \ref{prop:infinite-geodesic} depended only on the
asymptotic estimate for the total curvature of the pursuit curve,
whose continuous version is Theorem \ref{thm:sqrtestimate}, and on
Lemma \ref{lem:periodic}, which applies to any curve, so we
immediately obtain the continuous case:

\begin{prop}
Let $\Domain$ be a compact \catone\ space.   Suppose the evader
wins in continuous simple pursuit with initial distance
$L_0=d(P_0,E_0)<\pi$.  Then there is a bilaterally infinite local
geodesic in $\Domain$, any finite segment of which is the limit of
segments of the pursuit curve.
 \end{prop}


\section*{Acknowledgements}

\small This work is funded in part by DARPA \# HR0011-07-1-0002 and
by NSF MSPA-MCS \# 0528086. A preliminary and incomplete
version of this work appeared in the proceedings of the conference
{\em Robotics: Science and Systems 2006.} \normalsize


\appendix

\section{Background}

Here we set out the definitions and results that are assumed in the
paper. Further discussion may be found in \cite{BH99,BBI01}.

We consider {\em length spaces}, which are metric spaces for which
the distance between any two points is the infimum of pathlengths
joining them.

A curve $\sigma$ in a length space is a {\em geodesic} if
$d(\sigma(t),\sigma(t'))=|t-t'|$ for any two parameter values $t,t'$,
and a {\em local geodesic} if it is a geodesic when restricted to some
neighborhood of each of its parameter values. A length space is a
{\em geodesic space} if any two points are joined by a geodesic, and a
{\em C-geodesic space} if any two points with distance $<C$ are
joined by a geodesic.

Spaces with curvature bounded above are spaces whose geodesic
triangles are no `fatter' than triangles with the same sidelengths in a
model space of constant curvature, according to the following
definition.
\begin{definition}
A  geodesic metric space is \catzero\ if the distance between any two
points of any geodesic triangle $\triangle pqr$  is no greater than the
distance between the corresponding points of the {\em model
triangle} $\triangle \tilde{p}\tilde{q}\tilde{r}$  with the same
sidelengths in the Euclidean plane $M_0=\euc^2$.

A $(\pi/\sqrt{K})$-geodesic metric space is \catk\ for $K>0$ if  the
distance between any two points of any geodesic triangle $\triangle
pqr$ of perimeter  $< 2\pi/\sqrt{K}$ is no greater than the distance
between the corresponding points of the model triangle $\triangle
\tilde{p}\tilde{q}\tilde{r}$  with the same sidelengths in the
$2$-dimensional Euclidean sphere $M_K$ of radius
$1/\sqrt{K}$.\end{definition}

These definitions may be unified by setting $\pi/\sqrt{K}=\infty$ if
$K= 0$, as we will do from now on. For $K<0$,  \catk\ spaces are
defined similarly by taking $\pi/\sqrt{K}=\infty$ and $M_K$ to be the
hyperbolic plane of curvature $K$. \catk\ spaces for $K<0$ are
automatically \catzero\, and of no further interest in this paper.

Note that rescaling a \catone\ space by multiplying all distances by
$1/\sqrt{K}$, $K>0$, yields a  {\mbox{\sc cat($K$)}}\ space, since
rescaling $M_1$ yields $M_K$.

Since triangles with given sidelengths in the model spaces $M_K$
become fatter as $K$ increases,  it is clear that a {\mbox{\sc
cat($K_1$)}}\ space is also a {\mbox{\sc cat($K_2$)}}\ space for
$K_2>K_1$. It is an easy consequence of the definition that a
\catzero\ space $\Domain$ (respectively, a \catk\ space $\Domain$)
has unique geodesics between any two points (respectively, any two
points with distance $<\pi/\sqrt{K}$), and these geodesics vary
continously with their endpoints.  In particular, $\Domain$   is simply
connected (respectively, the open ball of radius $\pi/\sqrt{K}$ about
any point in $\Domain$  is simply connected).

For a simple example, take $\Domain$ to be the Euclidean plane with
one or more disjoint  open circular disks of radius $1$ removed, where
$\Domain$ is equipped with the length metric. Then $\Domain$ is not
a \catzero\ space but is \catone, since the \catone\ perimeter
condition excludes any triangle that encloses a removed disk, and the
remaining triangles are even thinner than their Euclidean models.
$\Domain$ is not simply connected, while open balls of radius $\pi$
are simply connected since they do not include any boundary circle.
(Since $\Domain$ is locally \catzero, its simply connected covering is
a  \catzero\ space, whose geodesics are sent to the local geodesics of
$\Domain$ by the covering map.)

An effective tool in  \catk\ geometry is {\em Reshetnyak
majorization}, which extends the defining comparison property of
\catk\ spaces:

\begin{theorem}[Reshetnyak \cite{Re68}]
\label{thm:Rmajor}
Let $\gamma$ be a closed curve of length $<2\pi/\sqrt{K}$ in a \catk\
space $X$. Then there is a closed curve $\widetilde\gamma$ which is
the boundary of a convex region $D$ in $M_K$ and a
distance-nonincreasing map $\varphi: D \to X$ such that the
restriction of $\varphi$ to $\widetilde\gamma$ is an
arclength-preserving map onto $\gamma$.
\end{theorem}

In a \catk\ space, the {\em angle} $\alpha\in [0,\pi]$ between two
geodesic segments starting from a common endpoint is well-defined:
it is the greatest lower bound of the corresponding angles in model
triangles for triangles formed by initial subsegments of the two
geodesics together with the attached third side. The thinness
condition implies that these model angles descend monotonically as
the subsegments are shortened, with the greatest lower bound equal
to the limit.

The angle between geodesic segments with a common origin gives the
\emph{angle distance} between the directions of geodesic segments
emanating from a point. In any \catk\ space, the angle distance
satisfies the triangle inequality.

The {\em First Variation Formula} for \catk\ spaces governs the rate
of change of distance $r(t)$ between unit-speed geodesics
$\gamma_1(t), \gamma_2(t)$, assuming $r < \pi/\sqrt{K}$:
\begin{equation}\label{first-variation}
    \left.\frac{dr}{dt}\right|_0 = -(\cos\alpha_1+\cos\alpha_2),
\end{equation}
where $\alpha_i$ is the angle between the geodesic $\gamma_i$ and
the geodesic joining $\gamma_1(0)$ and $\gamma_2(0)$. We can free
the First Variation Formula from the restriction that $r$ is the
distance between two unit speed geodesics, and assume only that the
two curves have righthand directions, by using the chain rule to insert
speed factors multiplying the terms.

\end{document}